\documentclass[11pt,twoside,reqno,centertags,english]{amsart}
\usepackage[oldstylenums]{kpfonts}
\usepackage{geometry}
\usepackage{graphicx}
\usepackage{lettrine}
\usepackage{amsmath,amssymb,verbatim}
\usepackage{enumerate}
\usepackage{hyperref}
\usepackage[active]{srcltx}
\usepackage[mathlines]{lineno}
\linenumbers
\newcommand{\mres}{\mathbin{\vrule height 1.6ex depth 0pt width
		0.13ex\vrule height 0.13ex depth 0pt width 1.3ex}}

\usepackage[T1]{fontenc}
\usepackage{color}

\newtheorem{theo}{Theorem}
\newtheorem{lem}{Lemma}[section]
\newtheorem{defi}[lem]{Definition}
\newtheorem{cor}[lem]{Corollary}
\newtheorem{prop}[lem]{Proposition}

\theoremstyle{remark}
\newtheorem{remark}[lem]{Remark}

\newcommand{\R}{\mathbb{R}}

\newcommand{\N}{\mathbb{N}}

\newcommand{\Mm}{\mathbb{P}^+}
\newcommand{\Mmm}{\mathbb{P}^{++}}
\newcommand{\Ma}{\mathbb{P}^1}

\newcommand{\Sm}{\mathcal{S}}
\newcommand{\Pm}{\mathcal{P}^{++}}
\newcommand{\Po}{{\mathcal{P}^+}}

\renewcommand{\o}{\overline}

\newcommand{\p}{\partial}
\numberwithin{equation}{section}
\DeclareMathOperator{\dive}{div}
\DeclareMathOperator{\grad}{grad}
\newcommand{\rhp}{G}
\newcommand{\rd}{\mathrm{d}}

\renewcommand{\u}{{\mathfrak{U}}}
\title[A new matrix distance]{On optimal transport of matrix-valued measures}
\author[Y.~Brenier]{Yann Brenier}
\address[Y.~Brenier]{D\'epartement de math\'ematiques et applications,
CNRS DMA (UMR 8553),
\'Ecole Normale Sup\'erieure,
45 rue d' Ulm,
75005 Paris,
France}
\email{yann.brenier@ens.fr}

\author[D.~Vorotnikov]{Dmitry Vorotnikov}
\address[D.~Vorotnikov]{CMUC, Department of
Mathematics, University of Coimbra, 3001-501 Coimbra, Portugal}{}
\email{mitvorot@mat.uc.pt}


\begin{document}
\maketitle
\begin{abstract}
We suggest a new way of defining optimal transport of positive-semidefinite matrix-valued measures. It is inspired by a recent rendering of the incompressible Euler equations and related conservative systems as concave maximization problems. The main object of our attention is the Kantorovich-Bures metric space, which is a matricial analogue of the Wasserstein and Hellinger-Kantorovich metric spaces. We establish some topological, metric and geometric properties of this space, which includes the existence of the optimal transportation path. 
\end{abstract}

%
\section{Introduction}

Positive-definite-matrix-valued densities arise in signal processing, geometry (Riemannian metrics) and other applications. Due to the success of the Monge-Kantorovich optimal transport theory, there have been recent attempts to introduce the matrix-valued optimal transport in a relevant way \cite{CGT18,CM14,CM17,CGT17,CGGT,CGT18A,PCS16,MM17,GMP16}. It is reasonable to try to achieve this goal via a dynamical formulation in line with \cite{BenamouBrenier00}. This requires  a kind of transport equation for matricial densities. The listed references either employ the Lindblad equaton \cite{GS} and related ideas, or have a static Monge-Kantorovich outlook.
Our approach is totally different, and we believe that it is promising for the applications because of its relative simplicity from the numerical perspective.

 The first author has recently observed  in \cite{B17eu} that the incompressible Euler equation can be recast as concave maximization problem. The method is actually applicable to various conservative PDEs, cf. \cite{B17eu,V19eu}. The procedure of \cite{B17eu,V19eu} naturally produces variational problems involving matricial densities. These problems are very similar to the dynamical optimal transport and to the mean-field games, and may serve a heuristic for constructing matricial optimal transport problems. The problem that we introduce in this paper is based on the transport-like operator \begin{equation}\label{symgr}- \left(\nabla q \right)^{Sym} \end{equation} that is related to the concave maximization rendering \cite{B17eu} of the incompressible Euler equation. Here $q$ is a suitable momentum-like field. One can generate \eqref{symgr} even more straigtforwardly by considering the Burgers-like problem \begin{equation}\label{burg} \partial_t v+\dive (v\otimes v)=0. \end{equation}  This is perhaps the most elementary vectorial PDE that fits into the framework of ``abstract Euler'' equations introduced in \cite{V19eu}. The corresponding concave maximization problem, cf.  \cite{V19eu}, may be formally written as \begin{equation} \label{e:absteu} \sup_{q,B}\int_{[0,T]\times \mathbb R^d} v_0\cdot q-\frac 1 4 G^{-1}q\cdot q\end{equation} where the vector fields $q$ and the positive-definite matrix fields $G$ are subject to the constraints
\begin{equation}\label{e:constr1}\p_t G_t=- \left(\nabla q_t \right)^{Sym},\end{equation} \begin{equation} G_T\equiv \frac 1 2 I.\end{equation}

However, the operator \eqref{symgr} that appears in \eqref{e:constr1} has a nontrivial cokernel, which means that we cannot join any two matrix-valued densities with a path directed by the tangents of the form \eqref{symgr}. This is somewhat similar to the impossibility of joining two measures of different mass in the classical Monge-Kantorovich transport. The latter issue can be fixed in the framework of the unbalanced optimal transport \cite{KMV16A,LMS16,CP18,LMS18,CPSV18,Rez15} by interpolating between the classical optimal transport and the  Hellinger (also known as Fisher-Rao) metric related to the information geometry \cite{Ay17,M15,KLMP}.  The matricial counterpart of the Hellinger metric is the Bures metric \cite{Bures,U92}. Notably, the corresponding Riemannian distance coincides with the Wasserstein distance between Gaussian measures \cite{BJL}. The Bures metric is usually defined for constant densities but can be naturally generalized to non-constant densities, cf. \cite{CGT18A}. Then we can interpolate between this quantum information metric and the matricial transport driven by \eqref{e:constr1}. This procedure generates an additional reactive term in the transport equation, and we can join any two positive-definite-matrix-valued measures by a suitable continuous path. The same correction term was recently used in \cite{CGT18A} for the Lindblad equation.
The resulting dynamical transportation problem generates a distance on the space of positive-definite-matrix-valued measures, which we call the Kantorovich-Bures distance. This distance is a matricial cognate of the Wasserstein distance \cite{villani03topics,S15} and of the recently introduced Hellinger-Kantorovich distance \cite{KMV16A,LMS16,CP18,LMS18,CPSV18}. The Kantorovich-Bures distance is frame-indifferent in the spirit of rational mechanics \cite{Tru}. The Kantorovich-Bures space is a geodesic metric space. It has a conic structure comparable to the one that was recently discovered \cite{LM17} for the Hellinger-Kantorovich space.  The Bures space of constant positive-definite matrices may be viewed as a totally geodesic submanifold in the Kantorovich-Bures space.

To finish the introduction, we would like to mention the recent works \cite{Forms1,Forms2,Forms3,Rez15,Forms4} aiming to launch a theory of optimal transport of differential forms. 

The paper is organized as follows. In the remaining part of the Introduction, we present basic notation and preliminary facts. In Section  \ref{section:metric_space}, we define the Bures-Kantorovich distance using a dynamical variational construction. In Section \ref{section:geodesic}, we explore some topological, metric and geometric properties of the Bures-Kantorovich metric space. In Section \ref{section:sphere}, we study the metric cone structure of the Bures-Kantorovich space. In the Appendices, we discuss the frame-indifference of the distance and formal Riemannian geometry of the Bures-Kantorovich space, and prove several technical lemmas.  

\subsection*{Notation and preliminaries}
\begin{itemize}
 \item 
 We will use the following basic notation: \begin{itemize}
 \item 
 $\R^{d\times d}
$
is the space of $d\times d$ matrices, equipped with the Frobenius product $\Phi:\Psi=Tr(\Phi\Psi^\top)$ and the norm $|\Phi|=\sqrt {\Phi:\Phi}$, \item $A^{Sym}:=\frac 1 2 (A+A^\top)$ will denote the symmetric part of $A\in \R^{d\times  d}$, \item $I\in \R^{d\times d}$ is the identity matrix,
\item $
\Sm
$ is the subspace of symmetric $d\times d$ matrices, \item
$
\Po
$ is the subspace of symmetric positive-semidefinite $d\times d$ matrices, \item
$
\Pm
$ is the subspace of symmetric positive-definite $d\times d$ matrices, \item
$
\mathcal P^1
$ is the subspace of symmetric positive-definite $d\times d$ matrices of unit trace, \item
%
$
\Mm
$ is the set of $\Po$-valued Radon measures $P$ on $\R^d$ with finite $Tr (\rd P(\R^d))$, \item
$
\Mmm
$ is the set of absolutely continuous (w.r.t. the Lebesgue measure $\mathcal L^d$) $\Pm$-valued Radon measures $P$ on $\R^d$ with finite $Tr (\rd P(\R^d))$,
\item $
\Ma
$ is the set of $\Po$-valued Radon measures $P$ on $\R^d$ with $Tr (\rd P(\R^d))$=1,
\item $Q:=[0,1]\times\R^d$.
\end{itemize}

\item
We will use the following simple inequalities
\begin{equation}\label{in:cs1}PA:A\leq Tr P|A|^2,\ Pq\cdot q\leq Tr P|q|^2,\ P\in \Po, A\in\R^{d\times d}, q\in \R^d.\end{equation}
\begin{equation}\label{in:cs2}A:B\geq 0,\ A,B\in \Po.\end{equation}
\item
We use the following notation for sets of functions (either scalar or matrix-valued):
\begin{center}
\begin{tabular}{ll}
$\mathcal{C}_b$: & bounded continuous with $\|\phi\|_{\infty}=\sup |\phi|$ \\ & (in the matrix-valued case the norm on the right is the Frobenius one); \\
$\mathcal{C}_b^1$: & bounded $\mathcal{C}^1$ with bounded first derivatives;\\
$\mathcal{C}^\infty_c$: & smooth compactly supported;\\
$\mathcal{C}_0$: & continuous and decaying at infinity;\\
$\operatorname{Lip}:$ & bounded and Lipschitz continuous with $\|\phi\|_{\operatorname{Lip}}=\|\nabla\phi\|_{\infty}+\|\phi\|_{\infty}$\\ & (here $\|\nabla\phi\|_{\infty}=\sup |\nabla\phi|$, where $|\cdot|$ is the operator norm in $\mathcal L(\R^d, \R^{d\times d})$\\ & for matrix-valued $\phi$).
\end{tabular}
\end{center}
\item
Given a sequence $\{\rhp^k\}_{k\in\N}\subset \Mm$ and $\rhp\in\Mm$ we say that:
\begin{enumerate}[(i)]
 \item $\rhp^k$ \emph{converges narrowly} to $\rhp$ if there holds
 $$
 \forall\,\phi\in \mathcal{C}_b(\R^d):\qquad \lim\limits_{k\to\infty}\int_{\R^d}\phi(x) \rd \rhp^k(x)=\int_{\R^d}\phi(x) \rd \rhp(x).
 $$
\item $\rhp^k$ \emph{converges weakly-$*$}  to $\rhp$ if there holds
 $$
 \forall\,\phi\in \mathcal{C}_0(\R^d):\qquad \lim\limits_{k\to\infty}\int_{\R^d}\phi(x) \rd \rhp^k(x)=\int_{\R^d}\phi(x) \rd \rhp(x).
 $$
\end{enumerate}

\item
For curves $t\in[0,1]\mapsto \rhp_t\in \Mm$ we write $\rhp\in\mathcal{C}_{w}([0,1];\Mm)$ for the continuity with respect to the narrow topology.
\item Clearly, $\Mm\subset(\mathcal{C}_0(\R^d;\mathcal S) )^*$. By approximating the constant matrix function identically equal to the unit matrix $I$ with compactly supported test functions, it is not difficult to prove  that \begin{equation} c\, Tr\,\rd\rhp(\R^d)\leq \|\rhp\|_{(\mathcal{C}_0(\R^d;\mathcal S) )^*}\leq C\, Tr\,\rd\rhp(\R^d) \label{masweak} \end{equation} for any $\rhp\in \Mm$, where the constants $c,C$ merely depend on $d$.

\item
Given a non-identically-zero measure $\rhp\in \Mm$ we will denote by $$L^2(\rd \rhp;\Sm\times \R^d)$$ the Hilbert space obtained by completion of the quotient by the seminorm kernel of the space  $\mathcal{C}^1_b(\R^d;\Sm\times \R^d)$ equipped with the Hilbert seminorm
$$
\|{\mathfrak U}\|_{L^2(\rd\rhp)}^2=\int_{\R^d}\rd \rhp(x)u\cdot u +\int_{\R^d} \rd \rhp(x)U(x):U(x).
$$ Here $\mathfrak U:=(U,u)$ stands for a generic element in $\mathcal{C}^1_b(\R^d;\Sm\times \R^d)$.

For brevity, we will simply write $L^2(\rd \rhp)$ instead of $L^2(\rd \rhp;\Sm\times \R^d)$. It is not difficult to see that the elements of $L^2(\rd \rhp)$ can be rendered as pairs ${\mathfrak U}=(U,u)\in L^2(\rd \rhp;\Sm)\times  L^2(\rd \rhp;\R^d)$, where the latter two spaces are defined in the conventional way (as, for instance, in \cite{DR97}).

\item 
In a similar fashion, 
given a narrowly continuous curve $\rhp\in \mathcal{C}_{w}([0,1];\Mm)$, we can define the space $L^2(0,1;L^2( \rd \rhp_t))$. The Hilbert norm in $L^2(0,1;L^2(\rd\rhp_t))$ is 
\begin{equation}
\|\u\|_{L^2(0,1;L^2(\rd\rhp_t))}^2=\int_0^1\left(\int_{\R^d}\rd \rhp_t(x)u_t (x)\cdot u_t(x) +\int_{\R^d} \rd \rhp_t(x)U_t(x):U_t(x)\right) \rd t.\label{e:lagr}
\end{equation}
\item The \emph{bounded-Lipschitz distance} ($BL$) between two matrix measures $\rhp_0,\rhp_1\in \Mm$ is
$$
d_{BL}(\rhp_0,\rhp_1)=\sup\limits_{\|\Phi\|_{\operatorname{Lip}}\leq 1}\left|\int_{\R^d}\Phi:(\rd \rhp_1-\rd \rhp_0)\right|.
$$
The distance $d_{BL}$ metrizes the narrow convergence on $\Mm$. A sketch of the proof in the case of matrix measures on an interval can be found in \cite{GN14}. In our situation the claim can still be shown by mimicking the proof strategy for the scalar-valued Radon measures \cite{B07,dud}. The key observation \cite{GN14} is that 
$\Sm$-valued bounded continuous functions can be approximated by monotone (in the sense of positive semi-definiteness) sequences of bounded Lipschitz ones.
We also point out is that the supremum can be restricted to smooth compactly supported functions. This follows from the tightness of a set consisting of two matricial measures of finite mass.
\item By \emph{geodesics} we always mean constant-speed, minimizing metric geodesics. 
\item $C$ is a generic positive constant.

\end{itemize}
%
\section{The Kantorovich-Bures distance}
The starting point for our considerations is
\label{section:metric_space}

\begin{defi}[Kantorovich-Bures distance] \label{d:metr}
Given two matrix measures $\rhp_0,\rhp_1\in \Mm$, we define
\begin{equation}\label{e:mini}
d_{KB}^2(\rhp_0,\rhp_1):=\inf_{\mathcal{A}(\rhp_0,\rhp_1)} \|\u\|^2_{L^2(0,1;L^2( \rd \rhp_t))},
\end{equation}
where the admissible set $\mathcal{A}(\rhp_0,\rhp_1)$ consists of all couples $(\rhp_t,\u_t)_{t\in [0,1]}$, ${\mathfrak U}_t=(U_t,u_t)$, such that
\begin{equation*}
\left\{ 
\begin{array}{l} 
\rhp\in\mathcal{C}_w([0,1];\Mm),\\
\rhp|_{t=0}=\rhp_0;\quad \rhp|_{t=1}=\rhp_1,\\
\u\in L^2(0,T;L^2(\rd\rhp_t)),\\
\p_t\rhp_t=\left\{-\nabla(\rhp_t u_t)+\rhp_tU_t\right\}^{Sym} \quad \mbox{in the weak sense, i.e.,}
\end{array}
\right.
\end{equation*}
\begin{multline} \label{consrt}
\int_{\R^d}\Phi_t:\rd \rhp_t-\int_{\R^d}\Phi_s:\rd \rhp_s-\int_s^t\int_{\R^d}(\rd\rhp_\tau:\p_\tau\Phi_\tau)\rd \tau\\
=
\int_s^t\int_{\R^d}(\rd\rhp_\tau u_\tau\cdot\dive\Phi_\tau+\rd\rhp_\tau U_\tau:\Phi_\tau)\rd \tau
\end{multline} for all test functions $\Phi\in \mathcal{C}^1_b(Q;\mathcal S)$ and $t,s\in [0,1].$
\end{defi}

We could have formally started from minimizing a more general  Lagrangian, namely, \begin{equation}\label{e:mini-1}
d_{KB}^2(\rhp_0,\rhp_1):=\inf_{\mathcal{B}(\rhp_0,\rhp_1)} \int_0^1\left(\int_{\R^d}\rhp_t^{-1}(x) q_t (x)\cdot q_t(x) +\rhp_t^{-1}(x)R_t(x):R_t(x)\,\rd x\right) \rd t,
\end{equation}
where the admissible set $\mathcal{B}(\rhp_0,\rhp_1)$ consists of  tuples $(\rhp_t,q_t,R_t)$, where $G_t(x)\in \Pm$, $q_t(x)\in \R^d$ and $R_t(x)\in \R^{d\times d}$, such that
\begin{equation*}
\left\{ 
\begin{array}{l} 
\rhp|_{t=0}=\rhp_0;\quad \rhp|_{t=1}=\rhp_1,\\
\p_t\rhp_t=\left\{-\nabla q_t+R_t\right\}^{Sym}.
\end{array}
\right. 
\end{equation*} This complies with \eqref{e:absteu}, \eqref{e:constr1} and the discussion in the Introduction. The reactive part is a generalization of the Bures metric, as will be evident in Remark \ref{constmatr}, see also Remark \ref{helling}. 

\begin{remark}In contrast with \eqref{e:absteu}, we opted for dropping the factor $1/4$ in the right-hand sides  of \eqref{e:mini} and \eqref{e:mini-1}, for a purely aesthetic reason, although this factor seems to be rather fundamental. Indeed, keeping it would halve the distance $d_{KB}$, which is in good agreement with Theorem \ref{th:cone} (ii). It would also eliminate the factor $4$ in Theorem \ref{prop:geodesic}, Proposition \ref{p:impr}, Corollary \ref{geodesictozero}, etc.\end{remark}

Perturbing an alleged minimizer of \eqref{e:mini-1} by adding $(\delta q,\delta R)$ for which $$L(\delta q,\delta R):=\left\{-\nabla (\delta q_t)+\delta R_t\right\}^{Sym}=0, (\delta q,\delta R)|_{t=0,1}=0,$$ we see that the minimizer formally satisfies $$\int_0^1\left(\int_{\R^d}\rhp_t^{-1}(x) q_t (x)\cdot \delta q_t(x) +\rhp_t^{-1}(x)R_t(x):\delta R_t(x)\,\rd x\right) \rd t=0, $$ for all pertubations $(\delta q,\delta R)$ from $Ker\, L$. This implies that such a minimizer can be written in the form $q=G\dive U$, $R=GU$, for some $U_t(x)\in\Sm$, hence \eqref{e:mini-1} yields \eqref{e:mini} via setting $u:=\dive U$. 
\begin{remark}[Transport of Hermitian matrices]    It seems that our results can be easily extended onto the case of matrix functions with complex entries; we opted for describing the real-valued case just for maintaining a more transparent connection with the classical Monge-Kantorovich transport. \end{remark}
\begin{remark}[Torus]  \label{tor} All the considerations of the paper are valid, mutatis mutandis, if the measures in question are defined on the flat torus $\mathbb T^d$ instead of $\mathbb R^d$.
\end{remark}
We shall prove shortly that
\begin{theo} 
 $d_{KB}$ is a distance on $\Mm$.
 \label{theo:d_distance}
\end{theo}
\noindent We first need a preliminary technical bound:
\begin{lem}
Let $\rhp\in \mathcal{C}_w([0,1];\Mm)$ be a narrowly continuous curve, assume that the constraint \eqref{consrt} is satisfied for some potential $\u\in L^2(0,T;L^2(\rd\rhp_t))$ with finite energy
$$
E=E[\rhp;\u]=\|\u\|^2_{L^2(0,T;L^2(\rd\rhp_t))}
$$
and let $M:=2(\max\{m_0,m_1\}+E)$ with $m_i:=Tr\,\rd\rhp_i(\R^d)$.
Then the masses are bounded uniformly in time, $$m_t:=Tr\, \rd\rhp_t(\R^d)\leq M,$$ and
\begin{equation}
\label{eq:fundamental_dBL_estimate}
\forall\, \Phi\in \mathcal{C}^1_b(\R^d;\Sm):\qquad \left|\int_{\R^d}\Phi:(\rd\rhp_t-\rd\rhp_s)\right| \leq (\|\dive \Phi\|_{\infty}+\|\Phi\|_{\infty})\sqrt{ME}|t-s|^{1/2}
\end{equation}
for all $0\leq s\leq t\leq 1$.
\label{lem:fundamental_dBL_estimate}
\end{lem}
\begin{proof}
Since $G_t$ is a narrowly continuous matrix function, the masses $m_t=\int_{\R^d}\rd \rhp_t :I$ are uniformly bounded, and $m:=\max\limits_{t\in [0,t]} m_t$ is finite.
Applying a Cauchy-Schwarz-like argument to the weak constraint \eqref{consrt}, and employing \eqref{in:cs1}, we deduce
\begin{multline*}
 \left|\int_{\R^d}\Phi:(\rd \rhp_t-\rd \rhp_s)\right|
 = \left|\int_{s}^{t}\left(\int_{\R^d}\rd \rhp_\tau u_\tau \cdot\dive \Phi +\int_{\R^d} \rd \rhp_\tau U_\tau :\Phi \right) \rd \tau\right| \\
  \leq\left(\int_{s}^{t}\left(\int_{\R^d}\rd \rhp_\tau \dive \Phi \cdot\dive \Phi +\int_{\R^d} \rd \rhp_\tau\Phi:\Phi\right) \rd \tau\right)^{1/2}\\ \times \left(\int_{s}^{t}\left(\int_{\R^d}\rd \rhp_\tau u_\tau \cdot u_\tau +\int_{\R^d} \rd \rhp_\tau U_\tau :U_\tau \right) \rd \tau\right)^{1/2}\\ \leq 
  ( \|\dive \Phi\|_{\infty} + \|\Phi\|_{\infty}) \sqrt{m}\cdot |t-s|^{1/2}E^{1/2},
\end{multline*}
and it is enough to estimate $m\leq M= 2(\max\{m_0,m_1\}+E)$ as in our statement.
Choosing $\Phi\equiv I$, the previous estimate yields $|m_t-m_s|\leq \sqrt{mE}|t-s|^{1/2}$. Let $t_0\in [0,1]$ be any time when $m_{t_0}=m$: choosing $t=t_0$ and $s=0$, we immediately get $m\leq m_0+\sqrt{mE}|t_0-0|^{1/2}\leq m_0+\sqrt{mE}$, and some elementary algebra bounds $m\leq 2(m_0+E)$. Exchanging the roles of $\rhp_0,\rhp_1$, we obtain $m\leq 2(m_1+E)$, and finally $m\leq M$.
\end{proof}

\begin{proof}[Proof of Theorem~\ref{theo:d_distance}]
Let us first show that $d_{KB}(\rhp_0,\rhp_1)$ is always finite for any $\rhp_0,\rhp_1\in \Mm$.
Indeed for any $P_0\in \Mm$ it is easy to see that $P_t=(1-t)^2P_0$ and $\u_t=\left(-\frac{2}{1-t}I,0\right)$ give a narrowly continuous curve $t\mapsto P_t\in \Mm$ connecting $P_0$ to zero, and an easy computation shows that this path has finite energy $E=4m_0<\infty$ (this curve is actually the geodesic between $P_0$ and $0$, see Corollary~\ref{geodesictozero} below). Rescaling time, it is then easy to connect any two measures $\rhp_0,\rhp_1\in \Mm$ in time $t\in [0,1]$ by first connecting $\rhp_0$ to $0$ in time $t\in [0,1/2]$ and then connecting $0$ to $\rhp_1$ in time $t\in [1/2,1]$ with cost exactly $E=8(m_0+m_1)<\infty$.

In order to show that $d_{KB}$ is really a distance, observe first that the symmetry $d_{KB}(\rhp_0,\rhp_1)=d_{KB}(\rhp_1,\rhp_0)$ is obvious by definition.

For the indiscernability, assume that $\rhp_0,\rhp_1\in \Mm$ are such that $d_{KB}(\rhp_0,\rhp_1)=0$.
Let $\left(\rhp^k_t,\u^k_t\right)_{t\in [0,1]}$ be any minimizing sequence in \eqref{e:mini}, i.e., $\lim\limits_{k\to \infty}E[\rhp^k;\u^k]=d_{KB}^2(\rhp_0,\rhp_1)=0$.
By Lemma~\ref{lem:fundamental_dBL_estimate} we see that the masses $m^k_t=Tr\,\rd\rhp^k_t(\R^d)$ are uniformly bounded: $$\sup\limits_{t\in [0,1],\,k\in\N}m^k_t\leq M.$$
For any fixed $\Phi\in \mathcal{C}^\infty_c(\R^d;\Sm)$ the fundamental estimate \eqref{eq:fundamental_dBL_estimate} gives
\begin{align*} 
\left|\int_{\R^d}\Phi:(\rd\rhp_1-\rd\rhp_0)\right| \leq (\|\dive \Phi\|_{\infty}+\|\Phi\|_{\infty})\sqrt{ME[\rhp^k;\u^k]}.
\end{align*}
Since $\lim\limits_{k\to\infty} E[\rhp^k;\u^k]=0$ we conclude that $\int_{\R^d}\Phi:(\rd\rhp_1-\rd\rhp_0)$ for all $\Phi\in \mathcal{C}^\infty_c(\R^d;\Sm)$, thus $\rhp_1=\rhp_0$ as desired.

As for the triangular inequality, fix any $\rhp_0,\rhp_1,P\in \Mm$ and let us prove that $d_{KB}(\rhp_0,\rhp_1)\leq d_{KB}(\rhp_0,P)+d_{KB}(P,\rhp_1)$.
We can assume that all three distances are nonzero, otherwise the triangular inequality trivially holds by the previous point.
Let now $(\underline{\rhp}^k_t,\underline{\u}^k_t)_{t\in[0,1]}$ be a minimizing sequence in the definition of $d_{KB}^2(\rhp_0,P)=\lim\limits_{k\to\infty}E[\underline{\rhp}^k;\underline \u^k]$, and let similarly $(\overline{\rhp}^k_t,\overline{\u}^k_t)_{t\in[0,1]}$ be such that $d_{KB}^2(P,\rhp_1)=\lim\limits_{k\to\infty}E[\overline{\rhp}^k;\overline \u^k]$.
For fixed $\tau\in (0,1)$ let $\left(\rhp_t,\u_t\right)$
be the continuous path obtained by first following $\left(\underline{\rhp}^k,\frac{1}{\tau}\underline{\u}^k\right)$ from $\rhp_0$ to $P$ in time $\tau$, and then following $\left(\overline{\rhp}^k,\frac{1}{1-\tau}\o \u^k\right)$ from $P$ to $\rhp_1$ in time $1-\tau$.
Then $\left(\rhp^k_t,\u^k_t\right)_{t\in[0,1]}$ is an admissible path connecting $\rhp_0$ to $\rhp_1$, hence by definition of our distance and the explicit time scaling (cf. \cite{KMV16A}) we get that
$$
d_{KB}^2(\rhp_0,\rhp_1)\leq E[\rhp^k;\u^k]=\frac{1}{\tau}E[\o \rhp^k;\o \u^k]+\frac{1}{1-\tau}E[\underline \rhp^k;\underline \u^k].
$$
Letting $k\to\infty$ we obtain for any fixed $\tau\in (0,1)$
$$
d_{KB}^2(\rhp_0,\rhp_1)\leq \frac{1}{\tau}d_{KB}^2(\rhp_0,P)+\frac{1}{1-\tau}d_{KB}^2(P,\rhp_1).
$$
Finally choosing $\tau=\frac{d_{KB}(\rhp_0,P)}{d_{KB}(\rhp_0,P)+d_{KB}(P,\rhp_1)}\in (0,1)$ yields
$$
d_{KB}^2(\rhp_0,\rhp_1)\leq\frac{1}{\tau}d_{KB}^2(\rhp_0,P)+\frac{1}{1-\tau}d_{KB}^2(P,\rhp_1)=(d_{KB}(\rhp_0,P)+d_{KB}(P,\rhp_1))^2
$$
and the proof is complete.
\end{proof}

\begin{cor}
\label{corm} The elements of a bounded set in $(\Mm,d_{KB})$ have uniformly bounded mass. Conversely, subsets of $\Mm$ with uniformly bounded mass are bounded in $(\Mm,d_{KB})$.
\end{cor}
\begin{proof} The first statement is  an easy consequence of Lemma~\ref{lem:fundamental_dBL_estimate}. The converse one follows from the observation that the squared distance from any element $P_0\in\Mm$ to zero is controlled by $4 m(P_0)$, see the proof of Theorem \ref{theo:d_distance}. 
\end{proof} Another simple property is 
\begin{lem} \label{l:hk}
 If $(\rhp_t,\u_t)_{t\in [0,1]}$ is a narrowly continuous curve with total energy $E$, then $t\mapsto \rhp_t$ is $1/2$-H\"older continuous w.r.t. $d_{KB}$, and more precisely
 $$
 \forall\, t_0,t_1\in [0,1]:\qquad d_{KB}(\rhp_{t_0},\rhp_{t_1})\leq \sqrt{E}|t_0-t_1|^{1/2}.
 $$
\end{lem}
\begin{proof}
Rescaling in time and connecting $\rhp_{t_0}$ to $\rhp_{t_1}$ by the path $$(\o\rhp_s;\o \u_s):=(\rhp_{t(s)},(t_1-t_0)\u_{t(s)}), \quad s\in [0,1],$$ with $t(s)=t_0+(t_1-t_0)s$, it is easy to see that the resulting energy scales as $$E|t_0-t_1| \geq E[\o\rhp;\o \u]\geq d_{KB}^2(\rhp_{t_0},\rhp_{t_1}).$$
\end{proof}

\begin{remark} \label{helling} In Definition \ref{d:metr} it is possible  to restrict ourselves to the admissible paths that satisfy the additional constraint $u\equiv 0$. This leads to another distance $d_H$ on $\Mm$, which is a matricial analogue of the Hellinger distance.  All the results of this paper remain true for $d_H$, and the proofs are literally the same. However, this distance is expected to be topologically stronger than $d_{KB}$ (see, however, Remark \ref{remtop}), which might be less relevant in applications. \end{remark}

\section{Properties of the distance and existence of geodesics} \label{section:geodesic}
We begin with some topological properties of the Kantorovich-Bures space. 
\begin{theo}[Comparison with narrow convergence]
\label{theo:equiv}
 The convergence of matrix measures w.r.t. the distance $d_{KB}$ implies narrow convergence, and any Cauchy sequence in $(\Mm,d_{KB})$ is Cauchy in $(\Mm,d_{BL})$.
Moreover, for any pair $\rhp_0,\rhp_1\in \Mm$ with masses $m_0,m_1$ there holds
\begin{equation}
 \label{eq:dBL_loc_equivalent_d}
 d_{BL}(\rhp_0,\rhp_1) \leq C_d\sqrt{(m_0+m_1)}d_{KB}(\rhp_0,\rhp_1)
\end{equation}
with some uniform $C_d$ depending only on the dimension. 
\end{theo}
\begin{proof} 
Fix $\rhp_0,\rhp_1$, and let $(\rhp_t,\u_t)$ be any admissible path from $\rhp_0$ to $\rhp_1$ with finite energy $E$.
Taking the supremum over $\Phi$ with $ \|\Phi\|_{\operatorname{Lip}}\leq 1$ in \eqref{eq:fundamental_dBL_estimate}, and observing that \begin{equation}\|\Phi\|_{\infty}+\|\dive \Phi\|_{\infty}\leq C \|\Phi\|_{\operatorname{Lip}},\label{ediv}\end{equation} where $C$ may merely depend on $d$, we deduce that $$d_{BL}(\rhp_0,\rhp_1)\leq C \sqrt{ME},$$
with $M=2(\max\{m_0,m_1\}+E)$ as in Lemma~\ref{lem:fundamental_dBL_estimate}.
Choosing now a minimizing sequence instead of an arbitrary path, and taking the limit we essentially obtain the same estimate with $E=\lim E[\rhp^k;\u^k]=d_{KB}^2(\rhp_0,\rhp_1)$. More precisely,
$$ 
d_{BL}(\rhp_0,\rhp_1)\leq C \sqrt{2(\max\{m_0,m_1\}+d_{KB}^2(\rhp_0,\rhp_1))}d_{KB}(\rhp_0,\rhp_1).
$$
Remembering the upper bound for $d_{KB}(\rhp_0,\rhp_1)$ in terms of the masses $m_0, m_1$ (cf.  Corollary~\ref{corm} or, explicitly, inequality \eqref{8ub}), we end up with \eqref{eq:dBL_loc_equivalent_d}.

Let $\rhp^k$ be a Cauchy sequence in $(\Mm,d_{KB})$ with mass $m^k=Tr\,\rd\rhp^k(\R^d)$. Since Cauchy sequences are bounded, we control the masses $m^k$ uniformly in $k$, thus from \eqref{eq:dBL_loc_equivalent_d} we see that
$$
d_{BL}(\rhp^p,\rhp^q)\leq  Cd_{KB}(\rhp^p,\rhp^q).
$$
Therefore, $\rhp^k$ is $d_{BL}$-Cauchy. Similarly, if a sequence is $d_{KB}$-converging, it is $d_{BL}$- and hence narrowly converging (to the same limit). 
\end{proof} 

\begin{remark} \label{remtop} The exact characterization of the topology of the Kantorovich-Bures space is an open problem. We also do not know whether this space is a complete metric space.
	\end{remark}

\begin{defi} Let $(X,\varrho)$ be a metric space, $\sigma$ be a Hausdorff topology on $X$. We say that the distance $\varrho$ is sequentially lower semicontinuous with respect to $\sigma$ if
for all $\sigma$-converging sequences
$x_k\overset{\sigma}{\rightarrow}x$, $y_k\overset{\sigma}{\rightarrow}y$
one has \begin{equation*}\label{e:lsq1}\varrho(x,y)\leq \liminf_{k\to \infty} \varrho(x_k,y_k).\end{equation*} \end{defi}

\begin{theo}[Lower-semicontinuity] \label{theo:d_LSC_weak*} The distance $d_{KB}$ is sequentially lower semicontinuous with respect to the weak-$*$ topology on $\Mm$.
\end{theo}

\begin{proof}
Consider any two converging sequences
$$
\rhp_0^k\underset{k\to\infty}{\rightarrow}\rhp_0,\qquad  \rhp_1^k\underset{k\to\infty}{\rightarrow}\rhp_1\qquad \mbox{weakly-}*
$$
of finite Radon measures from $\Mm(\R^d)$. For each $k$, the endpoints $\rhp_0^k$ and $ \rhp_1^k$ can be joined by an admissible narrowly continuous path $(\rhp^k_t,\u^k_t)_{t\in [0,1]}$ with energy
$$
E[\rhp^k;\u^k]\leq d_{KB}^2(\rhp_0^k,\rhp_1^k)+k^{-1}.
$$
Due to weak-$*$ compactness and \eqref{masweak}, the masses $m_0^k=Tr\,d\rhp^k_0(\R^d)$ and $m_1^k=Tr\,d\rhp^k_1(\R^d)$ are bounded uniformly in $k\in \N$. By Corollary \ref{corm}, the set $\cup_{k\in\N}\{\rhp_0^k,\rhp_1^k\}$ is bounded in $(\Mm,d_{KB})$, thus the energies $E[\rhp^k;\u^k]$ and the masses $m_t^k=Tr\,\rd\rhp^k_t(\R^d)$ are bounded uniformly in $k\in \N$ and $t\in [0,1]$:
$$
m_t^k\leq M
\quad\mbox{and}\quad
E[\rhp^k;\u^k]\leq \o E.
$$
Evoking \eqref{masweak} and the (classical) Banach-Alaoglu theorem with $\Mm\subset(\mathcal{C}_0(\R^d;\mathcal S) )^*$, we deduce that all the curves $(\rhp_t^k)_{t\in [0,1]}$ lie in a fixed weak-$*$ sequentially relatively compact set $\mathcal{K}_M=\{\rhp\in \Mm:\,Tr\,d\rhp(\R^d)\leq M\}$ uniformly in $k,t$. 
By the fundamental estimate \eqref{eq:fundamental_dBL_estimate} and \eqref{ediv} we get
$$
\left|\int_{\R^d}\Phi:(\rd\rhp^k_t-\rd\rhp^k_s)\right|\leq \sqrt{M \o E}|t-s|^{1/2}(\|\dive \Phi\|_{\infty}+\|\Phi\|_{\infty}) \leq C|t-s|^{1/2}(\|\nabla \Phi\|_{\infty}+\|\Phi\|_{\infty})
$$
for all $\phi\in \mathcal{C}^1_b$, which implies
$$
\forall\,t,s\in [0,1],\,\forall k\in \N:\qquad d_{BL}(\rhp^k_s,\rhp^k_t)\leq C|t-s|^{1/2}.
$$
This uniform $1/2$-H\"{o}lder continuity w.r.t. $d_{BL}$, the sequential lower semicontinuity of $d_{BL}$ with respect to the weak-$*$ convergence (Lemma~\ref{lem:d_BL_lsc_w*} in Appendix \ref{section:appendixa}), and the fact that $\rhp^k_t\in \mathcal{K}_M$ allow us to apply the refined version of Arzel\`a-Ascoli theorem (Lemma \ref{L:aa} in the Appendix \ref{section:appendixa}) to conclude that there exists a $d_{BL}$ (thus narrowly) continuous curve $(\rhp_t)_{t\in [0,1]}$ connecting $\rhp_0$ and $\rhp_1$ such that
\begin{equation}
\label{eq:pointwise_CV_w*}
\forall t\in [0,1]:\qquad \rhp^k_t\to \rhp_t\quad\mbox{ weakly-}*
\end{equation}
along some subsequence $k\to\infty$ (not relabeled here).
Let $\mu^k$ be the matricial measure on $Q$ defined by duality as
$$
\forall \,\phi\in \mathcal{C}_c(Q):\qquad \int_Q\phi(t,x)\rd\mu^k(t,x)=\int_0^1\left(\int_{\R^d}\phi(t,.)\rd\rhp^k_t\right)\rd t.
$$
Exploiting the pointwise convergence \eqref{eq:pointwise_CV_w*} and the uniform bound on the masses $m^k_t\leq M$, a simple application of Lebesgue's dominated convergence guarantees that
$$
\mu^k\to \mu^0\qquad \mbox{ weakly-}*\mbox{ in }{\Mm}(Q),
$$
where the finite measure $\mu^0\in \Mm(Q)$ is defined by duality in terms of the weak-$*$ limit $\rhp_t=\lim \rhp^k_t$ (as was $\mu^k$ in terms of $\rhp^k_t$), and, moreover, $$
\mu^k\mres [s,\tau]\times \R^d \to \mu^0\mres [s,\tau]\times \R^d\quad \mbox{ weakly-}*\mbox{ in }{\Mm}([s,\tau]\times \R^d),\quad s,\tau \in[0,1].
$$

Let $$
X\subset L^\infty(Q;\Sm\times \R^d)
$$ be the linear span of the functions of the form $$\Psi(t,x)=(\Phi,\phi)(t,x){\mathbf 1}_{[p,q]\times \R^d}(t,x), \quad (\Phi,\phi)\in \mathcal C^1_c(Q;\Sm\times \R^d),\quad p,q\in \mathbb Q\cap [0,1].$$
We are going to apply Proposition \ref{Ban} from Appendix \ref{section:appendixa} with this $X$ and the norm
$$
\|(\Phi,\phi)\|=\|\phi\|_{L^\infty(Q)}+\|\Phi\|_{L^\infty(Q)}.
$$ It is easy to see that $(X,\|\cdot\|)$ is separable (indeed, one can easily construct a dense set in $X$ equinumerous with the set of all tuples of natural numbers of discretionary finite length, which is countable).
Consider the following norms on $X$
$$\|(\Phi,\phi)\|_k=\left(\int_{Q}\rd\mu^k\phi\cdot \phi+\rd\mu^k \Phi:\Phi\right)^{1/2},\qquad k=0,1,\dots,
$$
and the linear forms
$$
\varphi_k(\Phi,\phi)=\int_{Q}\rd\mu^ku^k\cdot \phi+\rd\mu^k U^k :\Phi \,,\qquad k=1,2,\dots.
 $$ 
The weak-$*$ convergence of $\mu^k$, uniform boundedness of the masses of $Tr\,\mu^k(Q)\leq M$, and the Cauchy-Schwarz inequality imply that the hypotheses of Proposition \ref{Ban} are met with
$$
c_k:=\|\varphi_k\|_{(X,\|.\|_k)^*}\leq \|\u^k\|_{L^2(0,1;L^2(\rd\rhp^k))}= \sqrt{E[\rhp^k;\u^k]}\leq \sqrt{ d_{KB}^2(\rhp_0^k,\rhp_1^k)+k^{-1}}.
$$ 
Hence, there exists a continuous functional $\varphi_0$ on the space $(X,\|\cdot\|_0)$
 such that up to a subsequence 
$$
\forall (\Phi,\phi)\in\mathcal{C}^1_c(Q;\Sm\times\R^d):\qquad
\int_0^1\left(\int_{\R^d}\rd\rhp^k_t u^k_t\cdot \phi_t+\rd\rhp^k_t U^k_t :\Phi_t\right) \rd t \underset{k\to\infty}{\rightarrow} \varphi_0(\Phi,\phi)
$$
with moreover
\begin{equation}\label{e:phin}
\|\varphi_0\|_{(X,\|\cdot\|_0)^*}\leq \liminf_{k \to \infty} d_{KB}(\rhp_0^k,\rhp_1^k).
\end{equation}

Let $N_0\subset X$ be the kernel of the seminorm $\|\cdot\|_0$.
By the Riesz representation theorem, the dual $(X,\|\cdot\|_0)^*=(X/N_0,\|\cdot\|_0)^*$ can be isometrically identified with the completion $\overline {X/N_0}$ of ${X/N_0}$ with respect to  $\|\cdot\|_0$. As one can see, this completion is exactly $L^2(0,1;L^2(\rd\rhp_t))$.

Consequently, there exists $\u=(U,u)\in L^2(0,T;L^2(\rd\rhp_t))$ such that
$$
\varphi_0(\Phi,\phi)=\int_{Q}\rd\mu^0 u\cdot \phi+\rd\mu^0 U :\Phi =
\int_0^1\left(\int_{\R^d}\rd\rhp_t u_t\cdot \phi_t+\rd 
\rhp_t U_t :\Phi_t\right)\rd t
 $$
 and
 $$
 \|\u\|_{L^2(0,1;L^2(\rd\rhp_t))}=\|\varphi_0\|_{(X,\|\cdot\|_0)^*}.
 $$
Moreover, $(\rhp,\u)$ is an admissible curve joining $\rhp_0,\rhp_1$. Indeed, the established convergences are enough to pass to the limit in the constraint \eqref{consrt} with $t, s\in  \mathbb Q\cap[0,1]$ and $\Phi \in \mathcal C^2_c(Q; \mathcal S)$. Since $G_t$ is a narrowly continuous matrix function, an easy approximation argument shows that \eqref{consrt} actually holds for any $t, s\in [0,1]$ and $\Phi \in \mathcal C^1_b(Q; \mathcal S)$.

Recalling \eqref{e:phin}, it remains to take into account that by the definition of our distance
\begin{align*}
d_{KB}^2(\rhp_0,\rhp_1)\leq E[\rhp;\u]=\|\u\|^2_{L^2(0,1;L^2(\rd\rhp_t))}=\|\varphi_0\|_{(X,\|\cdot\|_0)^*}^2\leq \liminf\limits_{k\to\infty}d_{KB}^2(\rhp^k_0,\rhp^k_1).
\end{align*}
\end{proof}
During the proof of Theorem \ref{theo:equiv} we observed the upper bound \begin{equation} \label{8ub} d_{KB}^2(\rhp_0,\rhp_1)\leq 8(m_0+m_1).\end{equation} Let us show that it can improved.
\begin{prop}[Upper bound of the distance] \label{p:impr} For every pair $\rhp_0,\rhp_1\in \Mm$ with masses $m_0,m_1$ one has \begin{equation} \label{impr} d_{KB}^2(\rhp_0,\rhp_1)\leq  4(m_0+m_1). \end{equation}\end{prop} \begin{proof} Since $\Mmm$ is dense in $\Mm$ in the weak-$*$ topology (one can simply use the standard mollifiers), in view of Theorem \ref{theo:d_LSC_weak*} we can assume that $\rhp_0,\rhp_1\in \Mmm$. Consider the curve $$\rd G_t=\left(t\sqrt {G_1}+(1-t)\sqrt{G_0}\right)^2\rd \mathcal L^d.$$ The corresponding potential $\mathfrak U_t\in L^2(0,1;L^2(\rd G))$ can be defined by Riesz duality as \begin{equation} \label{fairpot} \langle{\mathfrak U},(\Phi,\phi)\rangle_{ L^2(0,1;L^2(\rd G))}=2\int_0^1\int_{\R^d} (\sqrt {G_1}-\sqrt {G_0}):\left(t\sqrt {G_1}+(1-t)\sqrt{G_0}\right)\Phi_t\,\rd \mathcal L^d\,\rd t\end{equation} for all $(\Phi,\phi)\in  L^2(0,1;L^2(\rd G))$. It is not difficult to see that the constraint \eqref{consrt} is satisfied. By definition, the energy of this path is $\|{\mathfrak U}\|^2_{ L^2(0,1;L^2(\rd G))}.$  
 By the Cauchy-Schwarz inequality and \eqref{in:cs2}, \begin{multline*}\langle{\mathfrak U},(\Phi,\phi)\rangle^2_{ L^2(0,1;L^2(\rd G))}  \\ \leq 4\int_0^1 \int_{\R^d}
   (\sqrt {G_1}-\sqrt {G_0}):(\sqrt {G_1}-\sqrt {G_0})\,\rd \mathcal L^d \rd t \\ \times \int_{\R^d} 
  \left(t\sqrt {G_1}+(1-t)\sqrt{G_0}\right)\Phi_t : \left(t\sqrt {G_1}+(1-t)\sqrt{G_0}\right)\Phi_t\,\rd \mathcal L^d\,\rd t \\ = 4\|(\Phi,0)\|^2_{ L^2(0,1;L^2(\rd G))}\int_0^1 \int_{\R^d}
 (\sqrt {G_1}-\sqrt {G_0}):(\sqrt {G_1}-\sqrt {G_0})\,\rd \mathcal L^d\,\rd t\\ \leq 4\|(\Phi,\phi)\|^2_{ L^2(0,1;L^2(\rd G))}\int_0^1 \int_{\R^d}
   (Tr\, {G_1}+Tr\,{G_0})\,\rd \mathcal L^d\,\rd t.  \end{multline*}  Thus, the energy  $E[G_t,{\mathfrak U}_t]$ is less than or equal to the right-hand side of \eqref{impr}.\end{proof}
\begin{remark} Corollary \ref{geodesictozero} indicates that the constant in \eqref{impr} is optimal. \end{remark}
We are now going to prove that $(\Mm,d_{KB})$ is a geodesic space.
%
\begin{defi} [cf. \cite{Bur}] We say that two points $x,y$ in a metric space $(X,\varrho)$ almost admit a midpoint if there exists a sequence $\{z_k\}\subset X$ such that 
\begin{equation*}
\label{amdp*}
|\varrho(x,y)-2 \varrho(x,z_k)|\leq k^{-1},\quad |\varrho(x,y)-2 \varrho(y,z_k)|\leq k^{-1}.
\end{equation*} 
\end{defi}
\begin{theo}[Existence of geodesics]
\label{theo:exist_geodesics}
$(\Mm,d_{KB})$ is a geodesic space, and for all $\rhp_0,\rhp_1\in \Mm$ the infimum in \eqref{e:mini} is always a minimum. Moreover this minimum is attained for a $d_{KB}$-Lipschitz curve $\rhp$ such that $d_{KB}(\rhp_t,\rhp_s)=|t-s|d_{KB}(\rhp_0,\rhp_1)$ and a potential $\u\in L^2(0,1;L^2(\rd\rhp_t))$ such that $\|\u_t\|_{L^2(\rd\rhp_t)}=cst=d_{KB}(\rhp_0,\rhp_1)$ for a.e. $t\in [0,1]$.
\end{theo}

\begin{proof}
We first observe from the definition of our distance that any two points in $\Mm$ almost admit a midpoint.  By Corollary \ref{corm} and the (classical) Banach-Alaoglu theorem,  $d_{KB}$-bounded sequences contain weakly-$*$ converging subsequences. Now Lemma \ref{hrt} (analogue of the Hopf-Rinow theorem for non-complete metric spaces) together with Theorem~\ref{theo:d_LSC_weak*} imply that $(\Mm,d_{KB})$ is a geodesic space. The existence and claimed properties of a minimizing admissible path in \eqref{e:mini} follow by mimicking the argument from the proof of Theorem~\ref{theo:d_LSC_weak*}  for the sequence of almost minimizing paths, and by evoking the general properties of metric geodesics \cite{AT,Bur}.
\end{proof}

\begin{remark} \label{divrem} It is possible to prove that the minimizing potential has the structure $$\mathfrak U_t=(U_t,\dive U_t)$$ in a certain generalized sense, in the spirit of \cite{KMV16A}, but this lies beyond the scope of this article.  \end{remark}
%

The next theorem gives some insight into the geometry of the Kantorovich-Bures space.
\begin{theo}[Explicit geodesics]
\label{prop:geodesic}
Fix any element $G_*\in \Mm$ and define the map $g:\Po\to \Mm$ by $$g(A)=A G_* A.$$ Then for any pair of commuting matrices $A_0,A_1\in \Po$ one has
\begin{equation}
\label{eq:distmatr}
d_{KB}^2(g(A_0), g(A_1))=4\int_{\R^d}\rd G_* (A_1-A_0):(A_1-A_0),
\end{equation}
and a geodesic between $g(A_0)$ and $g(A_1)$ is explicitly given by
\begin{equation}
\label{eq:geodmatr}
\bar G_t:=g(t A_1+(1-t)A_0).
\end{equation}
\end{theo}

\begin{proof} {\it Step 1}.  Define a potential $\bar{{\mathfrak U}}_t\in L^2(0,1;L^2(\rd\bar G))$ by Riesz duality as \begin{equation} \label{goodpot} \langle\bar{\mathfrak U},(\Phi,\phi)\rangle_{ L^2(0,1;L^2(\rd\bar G))}=2\int_0^1\int_{\R^d} \rd G_*(A_1-A_0):(t A_1+(1-t)A_0)\Phi_t\,\rd t\end{equation} for all $(\Phi,\phi)\in  L^2(0,1;L^2(\rd\bar G))$. A straightforward computation shows that $(\bar G_t, \bar{\mathfrak U}_t)$ satisfies the constraint \eqref{consrt}. The energy of this path coincides with $\|\bar{\mathfrak U}\|^2_{ L^2(0,1;L^2(\rd\bar G))}.$  
 By the Cauchy-Schwarz inequality, \begin{multline*}\langle\bar{\mathfrak U},(\Phi,\phi)\rangle^2_{ L^2(0,1;L^2(\rd\bar G))}  \\ \leq 4\int_0^1 \int_{\R^d}
  \rd  G_*(A_1-A_0):(A_1-A_0)\int_{\R^d} 
  \rd  G_*(t A_1+(1-t)A_0)\Phi_t: (t A_1+(1-t)A_0)\Phi_t\,\rd t \\\leq 4\|(\Phi,\phi)\|^2_{ L^2(0,1;L^2(\rd\bar G))}\int_0^1 \int_{\R^d}
  \rd  G_*(A_1-A_0):(A_1-A_0)\,\rd t.  \end{multline*} Thus, the energy  $E[\bar G_t,\bar {\mathfrak U}_t]$ is less than or equal to the right-hand side of \eqref{eq:distmatr}.

{\it Step 2}. In view of the previous step, it suffices to prove that the square of the distance is bounded from below by the right-hand side of \eqref{eq:distmatr}. We first observe that without loss of generality we may assume that $A_0\in \Pm$. Indeed, the general case $A_0\in \Po$ would immediately follow by letting $\epsilon \to 0_+$ in the triangle inequality $$d_{KB}(g(A_0), g(A_1))\geq d_{KB}(g(A_0+\epsilon I), g(A_1))-d_{KB}(g(A_0+\epsilon I), g(A_0)).$$

{\it Step 3}. Consider any admissible path $(\rhp_t,\u_t)_{t\in [0,1]}$ connecting $\rhp_0:=g(A_0)$ to $\rhp_1:=g(A_1)$. Let $\lambda$ be any scalar probability measure on $\R^d$. 
Set
$
\tilde{\rhp}_t:=\lambda\int_{\R^d}\rd\rhp_t, 
$ and define $\tilde{{\mathfrak U}}_t\in  L^2(0,1;L^2(\rd\tilde G))$ by duality as \begin{equation} \label{goodpot1} \langle\tilde{\mathfrak U},(\Phi,\phi)\rangle_{ L^2(0,1;L^2(\rd\tilde G))}=\int_0^1\int_{\R^d} \rd G_t U_t:\int_{\R^d} \Phi_t\rd \lambda\,\rd t\end{equation} for all $(\Phi,\phi)\in  L^2(0,1;L^2(\rd\tilde G))$.
Then$(\tilde {\rhp}, \tilde{\mathfrak U})$ is an admissible path (joining $A_0 \lambda \int_{\R^d} \rd G_* A_0$ and $A_1 \lambda \int_{\R^d} \rd G_* A_1$). We claim that it has lesser energy than $({\rhp}, {\mathfrak U})$. To prove the claim, we approximate this path with the sequence $$
\tilde{\rhp}_t^k:=\lambda\left(k^{-1} I+\int_{\R^d}\rd\rhp_t\right)
.$$  The corresponding potentials are \begin{equation} \label{goodpotk} \langle\tilde{\mathfrak U}^k,(\Phi,\phi)\rangle_{ L^2(0,1;L^2(\rd\tilde G^k))}=\int_0^1\int_{\R^d} \rd G_t U_t:\int_{\R^d} \Phi_t\rd \lambda\,\rd t.\end{equation} Let us equip the linear  space $\R^{d\times d}$ with the scalar product $$(B,B)_{k,t}=k^{-1}B:B+\int_{\R^d}\rd\rhp_t B:B,$$ and let $\Pi_{k,t}$ be the allied orthogonal projection onto the subspace $\mathcal S$. Then we can explicitly compute the components of the potentials $\tilde{\mathfrak U}^k$:
$$
\tilde{U}^k_t :=\Pi_{k,t}\left(\left(k^{-1}I+\int_{\R^d}\rd\rhp_t \right)^{-1}\int_{\R^d}\rd\rhp_t U_t\right),\qquad \tilde u^k_t\equiv 0.
$$ Performing some algebraic manipulations, we estimate
\begin{multline*}
E[\tilde{\rhp}^k;\tilde{\u}^k] 
  =\int_0^1 \int_{\R^d}\rd \lambda \left(k^{-1}I+\int_{\R^d}\rd\rhp_t\right) \tilde{U}^k_t: \tilde{U}^k_t\,\rd t = \int_0^1(\tilde{U}^k_t, \tilde{U}^k_t)_{k,t}\,\rd t\\ \leq \int_0^1 \left(\left(k^{-1}I+\int_{\R^d}\rd\rhp_t \right)^{-1}\int_{\R^d}\rd\rhp_t U_t, \left(k^{-1}I+\int_{\R^d}\rd\rhp_t \right)^{-1}\int_{\R^d}\rd\rhp_t U_t\right)_{k,t}\,\rd t \\ \leq \int_0^1 \left(k^{-1}I+\int_{\R^d}\rd\rhp_t \right)\left(\left(k^{-1}I+\int_{\R^d}\rd\rhp_t \right)^{-1}\int_{\R^d}\rd\rhp_t U_t\right)\\
  :\left( \left(k^{-1}I+\int_{\R^d}\rd\rhp_t \right)^{-1}\int_{\R^d}\rd\rhp_t U_t\right)\,\rd t\\+\int_0^1 \int_{\R^d}\rd\rhp_t\left(U_t-\left(k^{-1}I+\int_{\R^d}\rd\rhp_t \right)^{-1}\int_{\R^d}\rd\rhp_t U_t\right)\\:\left(U_t-\left(k^{-1}I+\int_{\R^d}\rd\rhp_t \right)^{-1}\int_{\R^d}\rd\rhp_t U_t\right)\,\rd t \\
  =  \int_0^1 \int_{\R^d}\rd\rhp_tU_t:U_t\,\rd t\\-k^{-1}\int_0^1 \left(\left(k^{-1}I+\int_{\R^d}\rd\rhp_t \right)^{-1}\int_{\R^d}\rd\rhp_t U_t\right):\left(\left(k^{-1}I+\int_{\R^d}\rd\rhp_t \right)^{-1}\int_{\R^d}\rd\rhp_t U_t\right)\,\rd t\\\leq E[\rhp;\u].
\end{multline*}
Arguing as in the proof of Theorem ~\ref{theo:d_LSC_weak*} we can pass to the limit inferior as $k\to \infty$ to show that $E[\tilde\rhp_t,\tilde{\u}_t]\leq E[\rhp;\u]$ as claimed. 

{\it Step 4}. Obviously, the right-hand side of \eqref{eq:distmatr} does not change if we replace $G_*$ by $\lambda D$, where  $D:=\int_{\R^d} \rd G_*\in \Po$, and $\lambda$ is as above. Thus, by the previous steps it is enough to check that the energies of the admissible paths of the form $\rhp_t=\lambda F_t$ with $F_t\in \Po$, $F_i=A_i D A_i$, $i=0,1$, $A_0\in \Pm$, with constant-in-space potentials $\mathfrak U_t=(U_t,0)\in L^2(0,1;L^2(\rd G))$ are bounded from below by the right-hand side of \eqref{eq:distmatr}. 
Some finite-dimensional calculus of variations shows that the minimum of those energies is achieved for $F_t=(t A_1+(1-t)A_0)D(t A_1+(1-t)A_0)$ with $U_t=2(t A_1+(1-t)A_0))^{-1}(A_1-A_0)$, $t\neq 1$. The corresponding energy is exactly $4D(A_1-A_0):(A_1-A_0)$. \end{proof}

\begin{cor}[Geodesic to zero]\label{geodesictozero} For any $G_*\in \Mm$,
$d_{KB}^2(G_*, 0)=4\, Tr\,\rd G_* (\R^d)$,
and $(1-t)^2 G_*$ is a geodesic between $G_*$ and $0$.
\end{cor}

\begin{remark}[Bures manifolds] \label{constmatr} The set $\Pm\subset \Sm$ has a natural structure of a smooth manifold, and the tangent space $T_{P}\Pm$ at every point $P\in \Pm$ can be identified with $\Sm$. For each $\Xi \in T_{P}\Pm$, let $U_{\Xi}\in \Sm$ be the unique solution to the Lyapunov equation \cite{Bh} 
$$2\Xi=PU_{\Xi}+U_{\Xi}P.$$ Then \begin{equation} \langle \Xi_1,\Xi_2\rangle_{P}:=PU_{\Xi_1}:U_{\Xi_2}\end{equation} is a Riemannian metric on $\Pm$ that is known as the Bures metric \cite{Bures,BJL}. The induced Riemannian metric on the submanifold $\mathcal P^1\subset \Pm$ is also called the Bures metric \cite{U92}. Actually, $\Pm$ is a metric cone over $\mathcal P^1$. In the next section we will see  that $\Mm$ has a similar  cone structure.   The geodesics between $P_0,P_1\in \Pm$ can be constructed as follows. Let $X\in \Pm$ be the unique solution to the Riccati equation \cite{Bh} $$XP_0X=P_1.$$ Then the geodesic is $$((1-t)I+tX)P_0((1-t)I+tX).$$ Let $d_B$ denote the corresponding Riemannian distance on $\Pm$. Fix any probability measure $\lambda $ on $\R^d$.  Since $I$ and $X$ commute, by Theorem \ref{prop:geodesic} the embedding $$P \mapsto P\lambda$$ from $\Pm$ into $\Mm$ is a totally geodesic map (in the sense of \cite{oh}).  Moreover, in view of Remark \ref{helling}, $$d_{B}(P_0,P_1)=d_{KB}(P_0 \lambda,P_1 \lambda)=d_{H}(P_0 \lambda,P_1 \lambda).$$ 


\end{remark}

\section{The spherical distance and the conic structure}

\label{section:sphere}

In this section we are going to explore the conic structure of $(\Mm,d_{KB})$. We start by defining a similar distance on $\Ma$ (analogue of probability measures) by simply normalizing the masses of the evolving densities:
\begin{defi} [Spherical Kantorovich-Bures distance]\label{d:metr2}
Given two matrix measures $\rhp_0,\rhp_1\in \Ma$ we define
\begin{equation}\label{e:mini2}
d_{SKB}^2(\rhp_0,\rhp_1):=\inf_{\mathcal{A}_1(\rhp_0,\rhp_1)}\int_0^1\left(\int_{\R^d}\rd \rhp_t u_t \cdot u_t +\int_{\R^d} \rd \rhp_t U_t :U_t \right) \rd t.
\end{equation}
where the admissible set $\mathcal{A}_1(\rhp_0,\rhp_1)$ consists of all couples $(\rhp_t,\u_t)_{t\in [0,1]}$ such that
\begin{equation*}
\left\{ 
\begin{array}{l} 
\rhp\in\mathcal{C}_w([0,1];\Ma),\\
\rhp|_{t=0}=\rhp_0;\quad \rhp|_{t=1}=\rhp_1,\\
\u\in L^2(0,T;L^2(\rd\rhp_t)),\\
\p_t\rhp_t=\left\{-\nabla(\rhp_t u_t)+\rhp_tU_t\right\}^{Sym} \quad \mbox{in the weak sense}.
\end{array}
\right.
\end{equation*}
\end{defi}

\begin{prop} $d_{SKB}$ is a distance on $\Ma$.\end{prop}

The proof is similar to the one of Theorem \ref{theo:d_distance}. Note that the indiscernability is obvious since by construction $d_{SKB}\geq d_{KB}$ on $\Ma$. 
\begin{remark}[Equivalent definition] \label{altsp} It is easy to see that Definition \ref{d:metr2} can be equivalently written in the following way: 
given two matrix measures $\rhp_0,\rhp_1\in \Ma$ we define
\begin{multline}\label{e:mini2.1}
d_{SKB}^2(\rhp_0,\rhp_1):=\inf_{\mathcal{A}_2(\rhp_0,\rhp_1)}\\ \int_0^1\left(\int_{\R^d}\rd \rhp_t u_t \cdot u_t  +\int_{\R^d} \rd \rhp_t U_t :U_t-\left(\int_{\R^d} \rd \rhp_t :U_t\right)^2\right) \rd t.
\end{multline}
where the admissible set $\mathcal{A}_2(\rhp_0,\rhp_1)$ consists of all couples $(\rhp_t,\u_t)_{t\in [0,1]}$ such that
\begin{equation*}
\left\{ 
\begin{array}{l} 
\rhp\in\mathcal{C}_w([0,1];\Ma),\\
\rhp|_{t=0}=\rhp_0;\quad \rhp|_{t=1}=\rhp_1,\\
\u\in L^2(0,T;L^2(\rd\rhp_t)),\\
\p_t\rhp_t+G_t\int_{\R^d} \rd G_t:U_t=\left\{-\nabla(\rhp_t u_t)+\rhp_tU_t\right\}^{Sym} \quad \mbox{in the weak sense}.
\end{array}
\right.\end{equation*} Indeed, $\mathcal{A}_1(G_0,G_1)=\mathcal{A}_2(G_0,G_1)\cap \left[\int_{\R^d} \rd G_t:U_t\equiv 0\right]$, hence the distance \eqref{e:mini2} is larger than or equal to \eqref{e:mini2.1}. On the other hand, the inverse inequality is also true since for any path $(\rhp_t,U_t,u_t)\in \mathcal{A}_2(G_0,G_1)$ we can find a path in $\mathcal{A}_1(G_0,G_1)$ of the same energy: one just takes $(\rhp_t,\mathfrak V_t),$ where $\mathfrak V_t\in L^2(0,T;L^2(\rd\rhp_t))$ is defined by duality via \begin{multline} \label{vpot} \langle {\mathfrak V},(\Phi,\phi)\rangle_{ L^2(0,1;L^2(\rd G))}
\\
=\int_0^1\left(\int_{\R^d}\rd \rhp_t(x) u_t (x)\cdot \phi_t(x) +
\int_{\R^d} \rd \rhp_t(x) \left[ U_t(x)-I\int_{\R^d} \rd 
\rhp_t(y):U_t(y)\right]:\Phi_t(x)\right)\,\rd t.
\end{multline} \end{remark}

We recall \cite{BH99,Bur} that, given a metric space $(X,d_X)$ of diameter $\leq \pi$, one can define another metric space $(\mathfrak C(X),d_{\mathfrak C(X)})$, called a \emph{cone} over $X$, in the following manner. Consider the quotient $\mathfrak C(X):=X\times [0,\infty)/X\times \{0\}$, that is, all points of the fiber $X\times \{0\}$ constitute a single point of the cone that is called the apex. Now set \begin{equation} \label{cone} d_{\mathfrak C(X)}^2([x_0,r_0],[x_1,r_1]):=r_0^2+r_1^2-2r_0r_1\cos (d_{X}(x_0,x_1)).\end{equation} Very few metric spaces are actually cones, and this property provides neat scaling and other nice geometric features \cite{LM17}. A particularly regular situation appears when the diameter of $X$ is strictly less than $\pi$, since in this case there is a one-to-one correspondence between the geodesics in $X$ and $\mathfrak C(X)$. Given a cone $Y=\mathfrak C(X)$, $X$ may be referred to as the \emph{sphere} in $Y$. 

\begin{lem}[Characterization of spherical distances] \label{charsp} If  $X$ is a length space, and $Y=\mathfrak C(X)$, then the distance $d_{X}(x_0,x_1)$ coincides with the infimum of $Y$-lengths of continuous curves $[x_t,1]$ that join $[x_0,1]$ and $[x_1,1]$ and lie within $X\times \{1\}$.  \end{lem}

\begin{proof}   Denote by $J(x_0,x_1)$ the infimum of $Y$-lengths of curves $[x_t,1]$  as in the statement of the lemma. Observe from \eqref{cone} that $d_{X}(x_+,x_-)\geq d_{\mathfrak C(X)}([x_+,1],[x_-,1])$ for any $x_+,x_-\in X$. Hence, the $Y$-length of any curve $[x_t,1]$ is less than or equal to the $X$-length of $x_t$. We claim that they are actually equal. It suffices to prove \begin{equation}\label{metL} L_Y([x_t,1])\geq q L_X(x_t)\end{equation} for any $q<1$. By linearity of \eqref{metL}, it is enough to prove it for curves of sufficiently small length.  From \cite[Ex. 3.6.14]{Bur} we infer  that \begin{equation*}L_Y([x_t,1])\geq 2 \sin\left( L_X(x_t)/2\right),\end{equation*} which yields \eqref{metL} for short curves. Since $X$ is a length space, we immediately conclude that  $J(x_0,x_1)= d_{X}(x_0,x_1)$. \end{proof}

We are going to show that the cone over the metric space $(\Ma,d_{SKB}/2)$ coincides with $(\Mm,d_{KB}/2)$. In other words, $(\Ma,d_{SKB}/2)$ is a sphere in the cone $(\Mm,d_{KB}/2)$, hence the name ``spherical distance''. Firstly, for any element $G\in \Mm$, we set $$r=r(G):=\sqrt{m(G)}=\sqrt{Tr\,\rd\rhp (\R^d)}.$$ Then we can identify $G$ with a pair $[G/r^2,r]\in \mathfrak C(\Ma)$. 

\begin{theo}[Conic structure] \label{th:cone} 
The space $(\Ma,d_{SKB})$ is a geodesic space of diameter $\le \pi$, while $(\Mm,d_{KB}/2)$ is a metric cone over $(\Ma,d_{SKB}/2)$, where $\Mm$ is identified with $\mathfrak C(\Ma)$ via $G\simeq [G/r^2,r]$. 
\end{theo}

\begin{proof} {\it Step 1.} We first observe that it suffices to show  the weaker claim that $(\Mm,d_{KB}/2)$ is  a metric cone over some metric space (which, due to the identification above, is nothing but $\Ma$ equipped with some distance $d$). Indeed, by Proposition \ref{p:impr}, for any two matrix measures $G_0,G_1\in \Ma$ one has \begin{equation} \label{hkbound} d_{KB}(G_0,G_1)/2\leq \sqrt 2.\end{equation} If $(\Mm,d_{KB}/2)$ is a cone over $(\Ma,d)$, \eqref{hkbound} and \eqref{cone} imply that $\cos(d(G_0,G_1))\geq 0$, whence the diameter of $(\Ma,d)$ is controlled from above by $\pi/2<\pi$. By Theorem \ref{theo:exist_geodesics} and \cite[Corollary 5.11]{BH99}, $(\Ma,d)$ is  a geodesic space. Evoking Lemma \ref{charsp} and Definition \ref{d:metr2}, we see that $d$ actually coincides with $d_{SKB}/2$.

 {\it Step 2.} In view of \eqref{hkbound} and \cite[Theorem 2.2]{LM17}, in order to prove the weaker claim discussed above it suffices to establish the following scaling property that characterizes the cones:
 \begin{equation}\label{scaling} d_{KB}^2(r_0^2G_0,r_1^2G_1)=r_0r_1d_{KB}^2(G_0,G_1)+4(r_0-r_1)^2,\end{equation} for all $G_0,G_1\in \Ma$, $r_0,r_1\geq 0$. Note that we have already proved it in the case $r_0r_1=0$ (see Corollary \ref{geodesictozero}), so we can assume that $r_0r_1>0$. Consider the scalar function $a(t)=\frac {r_1 t}{r_0+(r_1-r_0)t}$. Then $$a(0)=0,\ a(1)=1,\ a'(t)(r_0+(r_1-r_0)t)^2=r_0r_1.$$  We will also need its inverse function $t(a)$. 
 
 Let $(G_t,U_t,u_t)$ be any admissible path from $\mathcal A_1$ joining $G_0,G_1\in \Ma$. Then the path $(\tilde G_t,\tilde U_t, \tilde u_t)$, where \begin{gather*}\tilde G_t = (r_0+(r_1-r_0)t)^2 G_{a(t)},\\\tilde U_t=a'(t)U_{a(t)}+ \frac {2(r_1-r_0)}{r_0+(r_1-r_0)t}I,\\ \tilde u_t=a'(t)u_{a(t)}, 
 \end{gather*} connects $r_0^2G_0$ and $r_1^2G_1$. A straightforward computation shows that  $(\tilde G_t,\tilde U_t, \tilde u_t)$ satisfies the constraint \eqref{consrt}. Testing \eqref{consrt} with $\Phi_a=(r_0+(r_1-r_0)t(a))I$, we infer \begin{multline}\label{num4.8} (r_0+(r_1-r_0)t(1))\int_{\R^d} \rd \rhp_{1} : I -(r_0+(r_1-r_0)t(0))\int_{\R^d} \rd \rhp_{0} : I\\
 - (r_1-r_0)\int_0^1 t'(a)\int_{\R^d} \rd \rhp_{a} : I \,\rd a\\= \int_0^1(r_0+(r_1-r_0)t(a))\int_{\R^d} \rd \rhp_{a} : U_{a} \,\rd a. \end{multline}
 
 Let us compute the energy of the path $\tilde G_t$, employing \eqref{num4.8}: \begin{multline}E[\tilde G_t;\tilde U_t, \tilde u_t]=\int_0^1\left(\int_{\R^d}\rd \tilde\rhp_t\tilde u_t \cdot \tilde u_t +\int_{\R^d} \rd \tilde \rhp_t\tilde U_t:\tilde U_t\right) \rd t\\= r_0r_1\int_0^1a'(t)\left(\int_{\R^d}\rd \rhp_{a(t)} u_{a(t)} \cdot u_{a(t)}+\int_{\R^d} \rd \rhp_{a(t)} U_{a(t)}: U_{a(t)}\right) \,\rd t \\ + 4 (r_1-r_0)\int_0^1a'(t)(r_0+(r_1-r_0)t)\int_{\R^d} \rd \rhp_{a(t)} : U_{a(t)} \,\rd t\\+ 4 (r_1-r_0)^2\int_0^1\int_{\R^d} \rd \rhp_{a(t)} : I \,\rd t\\ =r_0r_1\int_0^1\left(\int_{\R^d}\rd \rhp_{a} u_{a} \cdot u_{a}+\int_{\R^d} \rd \rhp_{a} U_{a}: U_{a}\right) \,\rd a \\ + 4 (r_1-r_0)\int_0^1(r_0+(r_1-r_0)t(a))\int_{\R^d} \rd \rhp_{a} : U_{a} \,\rd a\\+ 4 (r_1-r_0)^2\int_0^1 t'(a)\int_{\R^d} \rd \rhp_{a} : I \,\rd a
 \\=r_0r_1 E[G_t;U_t, u_t]\\ +4 (r_1-r_0)(r_0+(r_1-r_0)t(1))\int_{\R^d} \rd \rhp_{1} : I\\ -4 (r_1-r_0)(r_0+(r_1-r_0)t(0))\int_{\R^d} \rd \rhp_{0} : I \\=r_0r_1E[G_t;U_t, u_t]+4(r_0-r_1)^2.\end{multline}
 
 Consequently,  $d_{KB}^2(r_0^2G_0,r_1^2G_1)\leq r_0r_1d_{KB}^2(G_0,G_1)+4(r_0-r_1)^2$. The opposite inequality is proved in a similar fashion.  \end{proof}
\appendix
\section{Frame-indifference} 
\label{find} 
The principle of material frame-indifference \cite{Tru} is one of the main
principles of rational mechanics, which expresses the fact that
the properties of a material do not depend on the choice of
an observer. An observer in rational mechanics is identified with a \emph{frame},
which is a correspondence between the spatial points and the
elements $x $ of the space $\mathbb{R}^d$, as well as between the
moments of time and the elements $\mathfrak t $ of the scalar axis
$\mathbb{R}$. The metrics in $\mathbb{R}^d$ and in the scalar axis, as well as the time direction, are assumed to be frame-invariant. Then the most
general change of coordinates is
\begin {gather*} \mathfrak t ^ * = \mathfrak t-\mathfrak t_0, \\ x ^ * = c^* (\mathfrak t) +Q_{\mathfrak t} x, \end {gather*} where $\mathfrak t_0\in\R$, $c ^
*: \R\to\R^d$,
$Q_{\mathfrak t}$ is a time-dependent orthogonal matrix.

Consider any vector that exists in the physical space irrespectively of the
observer. In the initial frame, it is represented by some $w\in \R^d$. Then in the new frame it is $w^*=Q_{\mathfrak t} w $.
A \emph{frame-indifferent tensor} is a linear automorphism of such vectors. The representations of a frame-indifferent tensor function in the two frames are related as
$$T ^ *(\mathfrak t^*,x^*) = Q_{\mathfrak t} T(\mathfrak t,x)Q_{\mathfrak t} ^\top.$$

We claim that our distance $d_{KB}$ complies with the frame indifference: \begin{equation} \label{frame-ind}d_{KB}\left(T^*_0(\mathfrak t^*,x^*), T^*_1(\mathfrak t^*,x^*)\right)=d_{KB}\left(T_0(\mathfrak t,x), T_1(\mathfrak t,x)\right).\end{equation} In other words, $d_{KB}$ may be considered as a distance on positive-semidefinite-frame-indifferent-tensor-valued measures. 

To prove the claim it suffices to note that for any admissible path $(T_t,U_t,u_t)(\mathfrak t,x)$ in the old frame, the path $$(T_t^*,U_t^*,u_t^*)(\mathfrak t^*,x^*):=\left(Q_{\mathfrak t} T_t(\mathfrak t,x)Q_{\mathfrak t} ^\top,Q_{\mathfrak t} U_t(\mathfrak t,x)Q_{\mathfrak t} ^\top, Q_{\mathfrak t} u_t(\mathfrak t,x)\right)$$ is admissible in the new frame, and has the same energy \eqref{e:mini}. These assertions can be verified by a straightforward computation: the only non-obvious issue for the validity of \eqref{consrt} in the new frame is that the spatial gradient is frame-indifferent:  
$$\nabla_{x^*} w^*=Q_{\mathfrak t}(\nabla_{x} w)Q_{\mathfrak t} ^\top$$ provided $w^*=Q_{\mathfrak t} w $,
which is just a manifestation of the chain rule, cf., e.g., \cite{Tru,ZV08}.
\section{Some technical facts}
\label{section:appendixa}

\begin{prop}[Refined Banach-Alaoglu \cite{KMV16A}]\label{Ban}
Let $(X,\|\cdot\|)$ be a separable normed vector space. Assume that there exists a sequence of seminorms $\{\|\cdot\|_k\}$ ($k=0,1,2,\dots$) on $X$ such that for every $x\in X$ one has 
$$
\|x\|_k\leq C \|x\|
$$
with a constant $C$ independent of $k,x$, and
$$
\|x\|_k\underset{k\to\infty}{\rightarrow} \|x\|_0.
$$
Let $\varphi_k$ ($k=1,2,\dots$) be a uniformly bounded sequence of linear continuous functionals on $(X,\|\cdot\|_k)$, resp., in the sense that 
$$
c_k:=\|\varphi_k\|_{(X,\|\cdot\|_k)^*}\leq C.
$$
Then the sequence $\{\varphi_k\}$ admits a converging subsequence $\varphi_{k_n}\to \varphi_0$ in the weak-$*$ topology of $X^*$, and
\begin{equation} \label{e:liminfp}
\|\varphi_0\|_{(X,\|\cdot\|_0)^*}\leq c_0:=\liminf\limits_k c_{k}.
\end{equation}
\end{prop} 

\begin{lem}\label{lem:d_BL_lsc_w*}
The matricial bounded-Lipschitz distance $d_{BL}$ is sequentially lower semicontinuous with respect to the weak-$*$ topology. 
\end{lem}

The proof is obvious since the supremum in the definition of $d_{BL}$ can be restricted to smooth compactly supported functions, which are dense in $\mathcal C_0$.

\begin{lem}[Refined Hopf-Rinow]
\label{hrt}
Let $(X,\varrho)$ be a metric space where every two points almost admit a midpoint. Assume that there exists a Hausdorff topology $\sigma$ on $X$ such that $\varrho$-bounded sequences contain $\sigma$-converging subsequences, and $\varrho$ is sequentially lower semicontinuous with respect to $\sigma$.  Then $(X,\varrho)$ is a geodesic space. 
\end{lem}
\begin{proof} Fix any two points $x_0,x_1\in X$. It suffices to join them by a curve $x_t$ such that \begin{equation}
\label{contcurve}
\varrho(x_t,x_{\bar t})\leq |t-\bar t|\varrho(x_0,x_1).
\end{equation} for all $t,\bar t\in[0,1]$ (which is a posteriori continuous). 

Let us first observe that every two points $x,y\in X$ admit a midpoint, that is, 
\begin{equation*}
\label{mdp}
\varrho(x,y)= 2 \varrho(x,z)=2\varrho(z,y).
\end{equation*} for some $z\in X$. 
Indeed, take any sequence $z_k$ of almost midpoints, i.e., 
\begin{equation*}
\label{amdp}
|\varrho(x,y)-2 \varrho(x,z_k)|\leq k^{-1},\quad |\varrho(x,y)-2 \varrho(y,z_k)|\leq k^{-1}.
\end{equation*} 
The sequence $\{z_k\}$ is $\varrho$-bounded, thus without loss of generality it $\sigma$-converges to some $z\in X$. Then $$2\varrho(x,z)\leq \lim_{k\to \infty} 2\varrho(x,z_k)= \varrho(x,y),$$
$$2\varrho(y,z)\leq \lim_{k\to \infty} 2\varrho(y,z_k)= \varrho(x,y).$$ But its is clear from the triangle inequality that the latter inequalities must be equalities. 

Let $Q=\{s\in[0,1]| \exists p\in \mathbb N:  2^p s\in \mathbb N\}$. With the existence of midpoints at hand, by a standard procedure \cite[p. 43]{Bur} one constructs points $x_s$ ($s\in Q$) satisfying \eqref{contcurve}, that is, the function $s\mapsto x_s$ is $\varrho(x_0,x_1)$-Lipschitz. Given any $t\in [0,1]$, we can approximate it by a sequence $\{s_n\}\in Q$. Since $s\mapsto x_s$ is Lipschitz on $Q$, $x_{s_n}$ is a $\varrho$-Cauchy sequence. Therefore it is $\varrho$-bounded, and admits a subsequence that $\sigma$-converges to some $x_t\in X$.  Due to the sequential lower semicontinuity of the distance $\varrho$, we can pass to the limit in \eqref{contcurve} for all $t,\bar t\in[0,1]$. \end{proof}

\begin{lem}[Refined Arzel\`a-Ascoli]
\label{L:aa}
Let $(X,\varrho)$ be a metric space. Assume that there exists a Hausdorff topology $\sigma$ on $X$ such that $\varrho$ is sequentially lower semicontinuous with respect to $\sigma$.  Let $(x^k)_t$, $t\in[0,1]$, be a sequence of curves lying in a common $\sigma$-sequentially compact set $K\subset X$.  Let it be equicontinuous in the sense that there exists a symmetric continuous function $\omega:[0,1]\times [0,1]\to\R_+$, $\omega(t,t)=0$, such that \begin{equation}
\label{eccurve}
\varrho((x^k)_t,(x^k)_{\bar t})\leq \omega(t,\bar t).
\end{equation} for all $t,\bar t\in[0,1]$. Then there exists a $\varrho$-continuous curve $x_t$ such that \begin{equation}
\label{eccurvelim}
\varrho(x_t,x_{\bar t})\leq \omega(t,\bar t),
\end{equation} and (up to a not relabelled subsequence) \begin{equation}
\label{conv}
(x^k)_t\to x_t 
\end{equation} for all $t\in[0,1]$ in the topology $\sigma$.
\end{lem} 
\begin{proof} A standard Arzel\`a-Ascoli argument allows us to construct, for each rational number $t\in [0,1]$, some points $x_t$ so that \eqref{conv} holds up to a not relabelled subsequence. Due to the sequential lower semicontinuity of $\varrho$, estimate \eqref{eccurvelim} is true for all rational $t,\bar t \in [0,1]$. Approximating any point $t\in [0,1]$ by a sequence of rational numbers, by mimicking the reasoning from the proof of  Lemma \ref{hrt} we can construct a $\varrho$-continuous curve $x_t$ satisfying \eqref{eccurvelim} for every $t\in [0,1]$. To show that the convergence \eqref{conv} takes place for all $t\in [0,1]$, one just repeats the argument from \cite[last part of the proof of Proposition 3.3.1]{AGS06}. \end{proof}
\begin{remark} Lemma \ref{hrt}  has been proved in \cite{KMV16A} assuming that $X$ is a complete length space, which is redundant. Similarly, Lemma \ref{L:aa} has been proved in \cite{AGS06} assuming that $X$ is a complete metric space. \end{remark}

\section{Matricial Otto calculus and beyond}

We have seen in Remark \ref{constmatr} that some pieces of $(\Mm,d_{KB})$ are isometric to Riemannian manifolds. One can (at least formally) extend this geometry onto the whole $\Mm$ such that the corresponding geodesic distance coincides with $d_{KB}$. Namely, we can develop some kind of Otto calculus, cf. \cite{otto01,villani03topics,villani08oldnew}, on $(\Mm,d_{KB})$. Starting from this point, we are completely formal. As we observed in Section \ref{section:metric_space} and in Remark \ref{divrem}, the minimizing potentials in \eqref{e:mini} can be chosen to be of the form $\u=(U,\dive U)$. This suggests to define the tangent spaces as
$$
T_\rhp\Mm:=\left\{\Xi=\left(-\nabla(\rhp\dive U)+\rhp U\right)^{Sym},\  U(x)\in\Sm\right\}
$$
and
$$
\|\Xi\|_{T_\rhp\Mm}=\|U\|_{H^1_{\dive}(\rd\rhp;\Sm)}:=\left(\int_{\R^d}\rd \rhp\dive U\cdot 
\dive U +\int_{\R^d} \rd \rhp U:U\right)^{1/2}.
$$
Ignoring all smoothness issues, the operator
\begin{equation} \label{e:epde}
\Xi(U)=\left(-\nabla(\rhp\dive U)+\rhp U\right)^{Sym}
\end{equation}
is $H^1_{\dive}(\rd\rhp;\Sm)$-coercive, so the one-to-one correspondence between the tangent vectors $\Xi$ and potentials $\u=(U,\dive U)$ is well defined.
By polarization this defines a Riemannian metric on $T\Mm$, and $$
d_{KB}^2(\rhp_0,\rhp_1)=\inf\left\{\int_0^1\left\|\frac{d\rhp_t}{dt}\right\|_{T_{\rhp_t}\Mm}^2\,\mathrm{d}t\right\}.
$$

The gradients of functionals $\mathcal{F}:\Mm\to\R$ are given by
\begin{equation}
\label{eq:formula_gradient_d}
 \grad_{KB}\mathcal{F}(\rhp)=\left[-\nabla\left(\rhp\dive\frac{\delta \mathcal{F}}{\delta\rhp}\right)+\rhp\frac{\delta \mathcal{F}}{\delta\rhp}\right]^{Sym},
\end{equation}
 where $\frac{\delta \mathcal{F}}{\delta\rhp}$ denotes the first variation with respect to the Euclidean structure of $L^2(\R^d,\Sm)$ with the standard scalar product $\langle U,U \rangle=\int_{\R^d} U:U$. The gradient flows are matricial PDEs of the form
$$
\p_t \rhp=\left[\nabla\left(\rhp\dive\frac{\delta \mathcal{F}}{\delta\rhp}\right)-\rhp\frac{\delta \mathcal{F}}{\delta\rhp}\right]^{Sym}.
$$
The interesting driving functionals include the von Neumann entropy $$\mathcal F_{N}(G)=\int_{\R^d} G\log G -G$$ and the ``volume'' $$\mathcal F_{V}(G)=\int_{\R^d} \sqrt{\det G}.$$ The gradient flow of $\mathcal F_{N}$ is a sort of matricial ``heat flow'' with logarithmic reaction. Indeed, for $d=1$ it simply becomes $\p_t G=\p_{xx}G- G\log G$. The gradient flow of $\mathcal F_{V}$ has some similarities with the mean curvature flow (if we view $G_t$ as an evolving Riemannian metric on $\R^d$). Unfortunately, it is not a genuinely geometric flow since the latter ones are expected to be invariant with respect to diffeomorphisms of $\R^d$ (for instance, the Ricci flow has this property), and our flow, in spite of the frame-indifference of the distance, does not behave in such a nice way.

A similar Otto calculus can be developed for the spherical space $(\Ma, d_{SKB})$. Remark \ref{altsp} guides us to define $$
T_\rhp\Ma:=\left\{\exists U(x)\in\Sm: \Xi=\left[-\nabla(\rhp\dive U)+\rhp U\right]^{Sym}-G\int_{\R^d} G:U\right\}
$$
and
$$
\|\Xi\|_{T_\rhp\Ma}=\left(\int_{\R^d}\rd \rhp\dive U\cdot 
\dive U +\int_{\R^d} \rd \rhp U:U-\left(\int_{\R^d} \rd \rhp:U\right)^2\right)^{1/2}.
$$

The gradients of functionals $\mathcal{F}:\Ma\to\R$ are 
\begin{equation}
 \grad_{SKB}\mathcal{F}(\rhp)=\left[-\nabla\left(\rhp\dive\frac{\delta \mathcal{F}}{\delta\rhp}\right)+\rhp\frac{\delta \mathcal{F}}{\delta\rhp}\right]^{Sym}-G\int_{\R^d} G:\frac{\delta \mathcal{F}}{\delta\rhp}.
\end{equation}

The second order calculus for both the cone and the sphere can be established by formally computing the geodesic equations, which leads to the definitions of Hessians and $\lambda$-convexity.

\begin{remark} 
	The original Otto calculus \cite{otto01} is related to the principal bundle structure of the space of diffeomeorphisms \cite{KLMP,M15,M17} for which the projection is Otto's Riemannian submersion. The base (horizontal space) of this bundle is the Otto-Wasserstein space (at least if we neglect the regularity issue), whereas the incompressible Euler equations (both homogeneous and inhomogeneous)  determine the geodesics on the fibers (vertical spaces) in the spirit of \cite{Ar66}. An analogous bundle construction recently discovered  in \cite{CPSV18,GV18} involves the Hellinger-Kantorovich space as the base and the multidimensional Camassa-Holm equations as the vertical geodesics.  We do not claim that there is a similar neat structure for which the (conic) Kantorovich-Bures space is the horizontal space. However, some weaker construction might exist since there is a natural candidate for the geodesic in the vertical space. 
	The following reasoning is very sloppy. Assume for simplicity that we work on the flat torus (cf. Remark \ref{tor}), and consider the ``vertical tangent vectors'' at $\rhp\simeq I\, \rd x$, where $\rd x$ stands for the canonical volume form on $\mathbb T^d$. 
	These vertical vectors should satisfy $$\left[-\nabla(\rhp u)+\rhp U\right] ^{Sym}=0,$$ that is, $$U=\left(\nabla u\right)^{Sym}.$$ The geodesic Lagrangian \eqref{e:lagr} on this fiber becomes \begin{multline}\int_{(0,1)\times \mathbb T^d} \left(u\cdot u + \left(\nabla u\right)^{Sym}:\left(\nabla u\right)^{Sym}\right) \rd x\, \rd t\\=\int_{(0,1)\times \mathbb T^d} \left(|u|^2 + \frac 1 2 |\nabla u|^2+\frac 1 2 |\dive u|^2 \right) \rd x\, \rd t.\end{multline} This is very similar to the Lagrangians of the EPDiff equations \cite{MM13,EPD05,EPD09,YO10,EPD13,EPD16,kolev} and of the geodesics for the $a$-$b$-$c$ metric \cite{KLMP} on the space of diffeomorphisms. 
\end{remark}

\subsection*{Acknowledgments}
Many of the results were obtained while the authors were participating in the Research in Pairs program at MFO in Oberwolfach, Germany. The authors thank F.-X. Vialard for interesting discussions.  DV was partially supported by the Portuguese Government through FCT/MCTES and by the ERDF through PT2020 (proj. UID/MAT/00324/2019, PTDC/MAT-PUR/28686/2017 and TUBITAK/0005/2014).

\subsection*{Conflict of interest statement} We have no conflict of interest to declare.
\bibliographystyle{abbrv}
\bibliography{biblio} 
\end{document}